\documentclass[11pt,leqno]{article}

\usepackage{amsfonts,latexsym,amsmath,amssymb,amsthm}
\usepackage{hyperref}
\usepackage{fullpage}
\usepackage{dsfont}
\newtheorem{theorem}{Theorem}[section]
\newtheorem{lemma}[theorem]{Lemma}
\newtheorem{proposition}[theorem]{Proposition}
\newtheorem{corollary}[theorem]{Corollary}
\newtheorem{remark}[theorem]{Remark}

\theoremstyle{definition}

\newtheorem*{note*}{Note}
\numberwithin{equation}{section}

\newcommand{\ls}{\leqslant}
\newcommand{\gr}{\geqslant}
\newcommand{\R}{\mathbb{R}}

\DeclareMathOperator{\vol}{vol}

\usepackage{setspace}
\begin{document}
\small

\title{\bf Geometry of the subgaussian body of an isotropic\\ convex body}

\author{Apostolos Giannopoulos, Minas Pafis and Natalia Tziotziou}

\date{}
\maketitle

\begin{abstract}\footnotesize 
For a centered convex body $K\subset\mathbb{R}^n$, let $\Psi_2(K)$ denote the symmetric convex body whose support function is given by the $\psi_2$-norms of 
linear functionals on $K$. The recent solution of Milman's problem on the existence of subgaussian directions by Letwin and Mikulincer 
naturally motivates the study of the geometry of this body.
We prove that $\Psi_2(K)$ has bounded volume ratio with respect to the $L_2$-centroid body $Z_2(K)$.
In the isotropic case, we also obtain sharp estimates for its mean width and the volume radii of its orthogonal projections, 
and derive consequences for the existence of subgaussian orthonormal bases. In particular, we construct orthonormal bases with quantitatively 
controlled subgaussian constants for every isotropic convex body.
\end{abstract}

%%%%%%%%%%%%%%%%%%%%%%%%%%%%%%%%%%%%%%%%%%%%%%%%%%%%%%%%%%%%%%%%%%%%%%%%%%%%%%%%%%%%%%%%%%%%%%%%%%%%%%%%%%%%%%%%%%%%%%%%%%%%%%%%%%%%%%%%%%%%%%%%%%%%%%%%%%%%%%%%%%%%%%%%%%%%
\section{Introduction}\label{section:1}
%%%%%%%%%%%%%%%%%%%%%%%%%%%%%%%%%%%%%%%%%%%%%%%%%%%%%%%%%%%%%%%%%%%%%%%%%%%%%%%%%%%%%%%%%%%%%%%%%%%%%%%%%%%%%%%%%%%%%%%%%%%%%%%%%%%%%%%%%%%%%%%%%%%%%%%%%%%%%%%%%%%%%%%%%%%%

Concentration properties of linear functionals on convex bodies play a central role in asymptotic convex geometry. In particular, directions along which linear functionals exhibit subgaussian behavior have been studied extensively since they provide strong control of moments and tail estimates. 

This motivates the following definition of a subgaussian direction. Let $K$ be a centered convex body in $\mathbb{R}^n$, that is, a convex body whose barycenter is the origin. Throughout the paper we assume that convex bodies
have volume $1$. A direction $v\in S^{n-1}$ is called {\it subgaussian} for $K$ with constant $b>0$ if
\begin{equation}\label{eq:subgaussian}\|\langle\cdot,v\rangle\|_{\psi_2}\ls b\,\|\langle\cdot,v\rangle\|_2.\end{equation}
Recall that the $\psi_2$-norm of a random variable $X$ on a probability space is defined by
\[\|X\|_{\psi_2}=\inf\left\{t>0:\mathbb E\exp\!\left((X/t)^2\right)\ls 2\right\}.\]
For bounded random variables $X$, 
\[\|X\|_{\psi_2}\approx\sup_{p\gr1}\frac{\|X\|_p}{\sqrt p}\approx\inf\left\{\alpha>0:\mathbb P(|X|>t)\ls2e^{-t^2/\alpha^2}\text{ for all }t\gr0\right\},\]
up to absolute constants. Uniform control of $\|X\|_p/\sqrt{p}$ is equivalent to a 
$\psi_2$-estimate, and hence to subgaussian tail decay (see, for example, \cite[Section~2.6]{Vershynin-book}). Throughout the paper, moments and 
$\psi_2$-norms of linear functionals are computed with respect to the uniform probability measure on $K$.

The following question was posed by V.~Milman.

\medskip

\noindent \textit{Question.} Does there exist an absolute constant $C>0$ such that every centered convex
body admits a subgaussian direction with constant $C$?

\medskip

Klartag~\cite{Klartag-2007} proved that every convex body admits a subgaussian direction, obtaining a polylogarithmic estimate for the constant in~\eqref{eq:subgaussian}. A different proof, yielding a slightly improved bound, was later given in~\cite{GPP07}. Subsequently, \cite{GPV11} showed that every convex body admits an $O(\sqrt{\ln(en)})$-subgaussian direction. Very recently, Letwin and Mikulincer~\cite{Letwin-Mikulincer-2026} answered Milman's question affirmatively. Their proof introduces several new ideas that play 
key roles in the present work.

While Milman's problem concerns the existence of a single subgaussian direction, its recent resolution naturally leads to the study of the geometry of the entire collection of such directions. This viewpoint is conveniently encoded in the \emph{subgaussian body}
$\Psi_2(K)$, defined as the symmetric convex body with support function
\[h_{\Psi_2(K)}(v)=\|\langle\cdot,v\rangle\|_{\psi_2}.\]
By definition, $\Psi_2(K)$ contains the ellipsoid $cZ_2(K)$, for an absolute constant $c>0$, where
\[h_{Z_2(K)}(v)=\|\langle\cdot,v\rangle\|_2.\]
It was proved in~\cite{GPV11} that
\[\left(\frac{{\rm vol}_n(\Psi_2(K))}{{\rm vol}_n(Z_2(K))}\right)^{1/n}\ls C\sqrt{\ln(en)},\]
where $C>0$ is an absolute constant.

Our first result strengthens this estimate by showing that the volume ratio of the subgaussian body with respect to $Z_2(K)$ 
is bounded by an absolute constant.

\begin{theorem}\label{th:vol-ratio-psi2-body}Let $K$ be a centered convex body in $\mathbb{R}^n$. Then
\[\left(\frac{{\rm vol}_n(\Psi_2(K))}{{\rm vol}_n(Z_2(K))}\right)^{1/n}\ls C,\]
where $C>0$ is an absolute constant.
\end{theorem}

Theorem~\ref{th:vol-ratio-psi2-body} and a standard argument imply the existence of a direction
$v\in S^{n-1}$ satisfying $h_{\Psi_2(K)}(v)\ls  Ch_{Z_2(K)}(v)$; equivalently,
\[\|\langle\cdot,v\rangle\|_{\psi_2}\ls C\,\|\langle\cdot,v\rangle\|_2.\]
Thus, Theorem~\ref{th:vol-ratio-psi2-body} provides another proof of the
affirmative answer to Milman's question. More importantly, however, it yields
a substantially stronger geometric statement by controlling the global size of
the body $\Psi_2(K)$ rather than merely guaranteeing the existence of a single
subgaussian direction.

Letwin and Mikulincer further proved in~\cite{Letwin-Mikulincer-2026}
that every centered convex body $K$ in $\mathbb{R}^n$
admits an orthonormal set
$\Theta=\{v_1,\ldots ,v_m\}\subset S^{n-1}$ with $|\Theta|=m\gr \frac{9n}{10}$ such that
\[c\sqrt p\,\|\langle\cdot,v\rangle\|_2\ls\|\langle\cdot,v\rangle\|_p\ls C\sqrt p\,\|\langle\cdot,v\rangle\|_2,\]
for every $v\in \Theta$, where the upper estimate holds for all $p\gr1$,
while the lower estimate is valid for $1\ls p\ls n$.
They also conjectured that, if $K$ is isotropic, then one may choose
$\Theta$ to be a complete orthonormal basis satisfying
\[c\sqrt p\ls\|\langle\cdot,v\rangle\|_p\ls C\sqrt p,\]
where again the upper estimate holds for all $p\gr1$, while the lower
estimate holds for $1\ls p\ls n$. In this case, $\|\langle\cdot,v\rangle\|_2=L_K\approx 1$ for every $v\in S^{n-1}$ (see Section~\ref{section:2} for
background information on isotropic convex bodies and the isotropic constant). Note that the upper estimate is equivalent to
\[\|\langle\cdot,v\rangle\|_{\psi_2}\ls C.\]

Our next theorem provides a quantitative partial answer to this conjecture.

\begin{theorem}\label{th:onb-psi2}
Let $K$ be an isotropic convex body in $\mathbb{R}^n$. Then there exists
an orthonormal basis $\{v_1,\ldots,v_n\}$ such that
\[\|\langle\cdot,v_k\rangle\|_{\psi_2}\ls C\sqrt{\frac{n}{\,n-k+1\,}}, \qquad 1\ls k\ls n,\]
where $C>0$ is an absolute constant.
\end{theorem}

The proof of Theorem~\ref{th:onb-psi2} is based on the projection estimate below, which is also of independent interest.

\begin{theorem}\label{th:vk-psi2-body}
Let $K$ be an isotropic convex body in $\mathbb{R}^n$. For every $1\ls m\ls n$ and any $F\in G_{n,m}$,
\begin{equation*}{\rm vrad}(P_F(\Psi_2(K)))\ls C\sqrt{n/m},\end{equation*} 
where $C>0$ is an absolute constant. 
\end{theorem}

The case $m=1$ shows that $\Psi_2(K)\subseteq C\sqrt{n}B_2^n$. The estimates of Theorem~\ref{th:vk-psi2-body} are sharp. Indeed, let
\[K_1={\rm vol}_n(B_1^n)^{-1/n}B_1^n,\]
the isotropic cross-polytope. In this case, as shown in \cite{Bobkov-Nazarov-2003}, 
\[\Psi_2(K_1)\approx\sqrt n\,B_1^n.\]
If $F_k$ is a coordinate $k$-dimensional subspace, then
\begin{align}\label{eq:sharp}
\max\bigl\{{\rm vrad}(P_F(\Psi_2(K_1))):\,F\in G_{n,k}\bigr\}
&\gr {\rm vrad}(P_{F_k}(\Psi_2(K_1)))\\
\nonumber &\approx \sqrt n\,{\rm vrad}(B_1^k)\approx\sqrt{\frac nk}.
\end{align}
In fact, Ivanov~\cite{Ivanov-2026} recently confirmed the conjecture that
among all $k$-dimensional projections of the cross-polytope, the coordinate
ones have maximal volume. Consequently, the first inequality in
\eqref{eq:sharp} is actually an equality.

Using the recent optimal $M$-estimate of Bizeul \cite{Bizeul-M-2026} for isotropic convex bodies,
we obtain a sharp upper bound for the mean width of the subgaussian body.

\begin{theorem}\label{th:mean-width}Let $K$ be an isotropic convex body in ${\mathbb{R}}^n$. Then,
\[w(\Psi_2(K))\ls C\sqrt{\ln(en)},\]
where $C>0$ is an absolute constant.
\end{theorem}

Brazitikos and Hioni \cite{Brazitikos-Hioni-2015} had previously obtained the upper bound $w(\Psi_2(K))\ls C(\ln(en))^2$. Note that
$$w(\Psi_2(K_1))\approx w(\sqrt{n}B_1^n)\approx\sqrt{\ln(en)}$$
for the isotropic cross-polytope $K_1$, hence the estimate of Theorem~\ref{th:mean-width} is asymptotically sharp. 

A further simple observation is that \[\Psi_2(K)\supseteq c\,Z_2(K)=cL_KB_2^n\]
for every isotropic convex body $K$ in $\mathbb{R}^n$, which implies that
\[M(\Psi_2(K))=\int_{S^{n-1}}\|\xi\|_{\Psi_2(K)}\,d\sigma(\xi)\ls C.\]
Combining this estimate with Theorem~\ref{th:mean-width}, we obtain
\[M(\Psi_2(K))\,w(\Psi_2(K))\ls C\sqrt{\ln(en)}.\]
Thus, in the terminology of \cite[Chapter~6]{AGA-book}, the subgaussian body of an isotropic convex body is in the $\ell$-position.

\smallskip 

We next return to the conjecture of Letwin and Mikulincer. We obtain a complementary result by showing that one can always find a complete orthonormal
basis consisting entirely of subgaussian directions, with subgaussian constants growing at most logarithmically in the dimension. More
precisely, we prove the following theorem.

\begin{theorem}\label{th:partial-onb}Let $K$ be an isotropic convex body in ${\mathbb{R}}^n$. There exists
an orthonormal basis $\{v_1,\ldots ,v_n\}$ of $\mathbb{R}^n$ such that
\begin{equation*}\|\langle \cdot ,v_i\rangle\|_{\psi_2}\ls C\,\sqrt{\ln(en)}\end{equation*}
for all $i=1,\ldots ,n$. In fact,
$$\int_{O(n)}\max_{1\ls i\ls n} \|\langle \cdot,U(e_i)\rangle \|_{\psi_2}d\nu(U)  \ls C_1\sqrt{\ln (en)},$$
where $\nu$ is the Haar probability measure on $O(n)$ and $\{e_1,\ldots ,e_n\}$ is the standard orthonormal basis of $\mathbb{R}^n$.
\end{theorem}

Theorem~\ref{th:partial-onb} should be viewed as complementary to Theorem~\ref{th:onb-psi2}. The latter provides a basis in which the
subgaussian constants are uniformly bounded for most vectors and deteriorate only near the end of the basis, whereas
Theorem~\ref{th:partial-onb} yields a complete orthonormal basis whose subgaussian constants are bounded logarithmically in the dimension.

We conclude by identifying an important class of convex bodies for which the preceding methods yield optimal results. In particular, we prove that the Letwin–Mikulincer conjecture holds for every unconditional isotropic convex body. The proof combines our general approach with the sharp $\psi_2$-estimates of Bobkov and Nazarov.

\smallskip

The paper is organized as follows. In Section~\ref{section:2} we collect the necessary background and auxiliary results. 
Section~\ref{section:3} is devoted to the proofs of our main results: the volume estimate for the subgaussian body, 
the projection theorem and its consequences, the construction of orthonormal bases of subgaussian directions, and the unconditional case.
Section~\ref{section:4} contains further remarks and questions.

%%%%%%%%%%%%%%%%%%%%%%%%%%%%%%%%%%%%%%%%%%%%%%%%%%%%%%%%%%%%%%%%%%%%%%%%%%%%%%%%%%%%%%%%%%%%%%%%%%%%%%%%%%%%%%%%%%%%%%%%%%%%%%%%%%%%%%%%%%%%%
\section{Background and auxiliary facts}\label{section:2}
%%%%%%%%%%%%%%%%%%%%%%%%%%%%%%%%%%%%%%%%%%%%%%%%%%%%%%%%%%%%%%%%%%%%%%%%%%%%%%%%%%%%%%%%%%%%%%%%%%%%%%%%%%%%%%%%%%%%%%%%%%%%%%%%%%%%%%%%%%%%%%

We work in $\mathbb{R}^n$ equipped with the standard inner product $\langle \cdot, \cdot \rangle$.  
The associated Euclidean norm is denoted by $|\cdot|$, and $B_2^n$ and $S^{n-1}$ denote the Euclidean unit ball and unit sphere, respectively.  
Lebesgue measure on $\mathbb{R}^n$ is denoted by ${\rm vol}_n$, and we write $\omega_n = {\rm vol}_n(B_2^n)$.
The rotationally invariant probability measure on $S^{n-1}$ is denoted by $\sigma$, and $\nu$ is the Haar probability measure on $O(n)$ . 
The Grassmannian of $k$-dimensional subspaces of ${\mathbb{R}}^n$ is denoted by $G_{n,k}$ and is equipped with the Haar probability measure $\nu_{n,k}$. 

Throughout the paper, the symbols $C, c, c', c_1, c_2, \ldots$ denote absolute positive constants whose values may change from line to line.  
Whenever we write $a \approx b$ for positive quantities $a,b$, we mean that there exist absolute constants $c_1, c_2 > 0$ such that $c_1 a \ls b \ls c_2 a$. 
For convex bodies, $A\approx B$ means that $c_1A\subseteq B\subseteq c_2A$.

For background on isotropic convex bodies and log-concave measures, we refer the reader to~\cite{BGVV-book};  
for general background on asymptotic geometric analysis, see~\cite{AGA-book, AGA-book-2}.

\medskip

A {convex body} in $\mathbb{R}^n$ is a compact convex subset $K$ with nonempty interior.  
We say that $K$ is {symmetric} if $K = -K$, and {centered} if its barycenter
\[{\rm bar}(K)=\frac{1}{{\rm vol}_n(K)}\int_K x\,dx\] is the origin.

Let $K$ be a convex body with $0 \in {\rm int}(K)$. Its radial function is defined by 
$\rho_K(u) = \max\{ t > 0 : tu \in K \}$ for $u \in S^{n-1}$, 
and its {support function} is
$h_K(y) = \max \{ \langle x, y \rangle : x \in K \}$ for $y\in\mathbb{R}^n$. The mean width of $K$ is the average of its support function on $S^{n-1}$:
\[w(K)=\int_{S^{n-1}}h_K(\xi)\,d\sigma(\xi).\]
We note that if $T \in GL_n$, then
\[\frac{1}{\|T^{-1}\|_{{\rm op}}} w(K) \ls w(TK) \ls \|T\|_{{\rm op}} w(K).\]

The Minkowski functional of $K$ is $p_K(x) = \inf\{ t > 0 : x \in tK \}$. The mean norm of $K$ is the average of its Minkowski functional on $S^{n-1}$:
\[M(K)=\int_{S^{n-1}}p_K(\xi)\,d\sigma(\xi).\]
The outer radius, inradius and volume radius of $K$ are respectively $R(K) = \max\{ |x| : x \in K \}$, $r(K)=\max\{r>0:rB_2^n\subseteq K\}$,
and
\[{\rm vrad}(K) = \left( \frac{{\rm vol}_n(K)}{{\rm vol}_n(B_2^n)} \right)^{1/n}.\]
The {polar body} of $K$ is
\[K^{\circ} = \{ x \in \mathbb{R}^n : \langle x, y \rangle \ls 1 \text{ for all } y \in K \}.\]
For every convex body $K\subseteq \mathbb{R}^n$ we write $\overline{K}$ for the multiple of $K$ that has
volume $1$; in other words, $\overline{K}:={\rm vol}_n(K)^{-1/n}K$.

\medskip
An absolutely continuous probability measure $\mu$ on $\mathbb{R}^n$ is called {log-concave} if its density with respect
to Lebesgue measure is of the form
$f_\mu = e^{-\varphi}$, where $\varphi : \mathbb{R}^n \to \mathbb{R} \cup \{+\infty\}$ is convex.  
The uniform probability measure on a convex body is log-concave. The barycenter of a probability measure $\mu$ is
\[{\rm bar}(\mu) := \int_{\mathbb{R}^n} x f_\mu(x) \, dx,\]
and its {isotropic constant} is the quantity
\begin{equation*}L_\mu := \|f_\mu\|_{\infty}^{\frac{1}{n}}\det({\rm Cov}(\mu))^{\frac{1}{2n}},\end{equation*}
where
\[{\rm Cov}(\mu) := \int x \otimes x \,f_{\mu}(x)\, dx- \left( \int xf_{\mu}(x) \, dx \right) \otimes \left( \int xf_{\mu}(x) \, dx \right)\]
is the covariance matrix of $\mu$.  
A log-concave probability measure $\mu$ on $\mathbb{R}^n$ is called {isotropic} if ${\rm bar}(\mu) = 0$ and ${\rm Cov}(\mu) = I_n$. For an isotropic
log-concave probability measure, $L_{\mu}=\|f_{\mu}\|_{\infty}^{1/n}$ since $\det({\rm Cov}(\mu))=1$. 
\medskip

A convex body $K \subset \mathbb{R}^n$ is called {isotropic} if it has volume $1$, its barycenter is at the origin, and there exists a constant $L_K > 0$ such that
\[\int_K \langle x, \xi \rangle^2 \, dx = L_K^2\quad\text{for all } \xi \in S^{n-1}.\]
Every convex body $K$ admits an isotropic affine image, unique up to orthogonal transformations (see \cite{BGVV-book}).
The constant $L_K$ is an affine invariant of $K$. 

A central problem in asymptotic convex geometry, posed by Bourgain~\cite{Bourgain-1986},
is the hyperplane conjecture: whether there exists an absolute constant $C>0$
such that
\[L_n:=\max\{L_K:K\text{ is an isotropic convex body in }\mathbb{R}^n\}\ls C\]
for every $n\gr 1$. In fact, Ball~\cite{Ball-1988} showed that, in every
dimension,
\[\sup_\mu L_\mu\ls C\sup_K L_K,\]
where the suprema are taken over all log-concave probability measures $\mu$ and
all convex bodies $K\subseteq\mathbb{R}^n$, respectively.

The hyperplane conjecture was recently resolved affirmatively by
Klartag and Lehec~\cite{KL}, following major progress by Guan~\cite{Guan};
an alternative proof was subsequently obtained by Bizeul~\cite{Bizeul-2025}.
The latter approach relies on the following optimal small-ball estimate:
if $\mu$ is an isotropic log-concave probability measure on $\mathbb{R}^n$,
then for every $0<\varepsilon\ls c_0$ and every $y\in\mathbb{R}^n$,
\begin{equation}\label{eq:optimal-small-ball}
\mu\left(\{x\in\mathbb{R}^n:|x-y|^2\ls\varepsilon n\}\right)
\ls \varepsilon^{c_0 n},
\end{equation}
where $c_0>0$ is an absolute constant.

E.~Milman proved in \cite{EMilman-2014} that if $K$ is an isotropic convex body in ${\mathbb R}^n$, then
$$w(K) \ls C \sqrt{n} (\ln(en))^2 L_K \ls c_1 \sqrt{n} (\ln (en))^2,$$
where the second inequality uses the boundedness of $L_n$. Very recently, Bizeul \cite{Bizeul-M-2026} obtained the optimal upper bound
\begin{equation}\label{eq:Bizeul-M}
w(K) \ls C\sqrt{n\ln(en)}
\end{equation}
for any isotropic convex body $K$ in ${\mathbb R}^n$.

\medskip 

Let $\mu$ be a centered log-concave probability measure on $\mathbb{R}^n$. For $p\gr 1$, the
$L_p$-centroid body $Z_p(\mu)$ of $\mu $ is the symmetric convex body whose support function is
\begin{equation*} h_{Z_p(\mu)}(y):=\|\langle \cdot ,y\rangle\|_{L_p(\mu)}=\left(
\int_{\mathbb{R}^n} |\langle x,y\rangle|^p f_{\mu}(x)dx \right)^{1/p},\qquad y\in \mathbb{R}^n.
\end{equation*} The centroid bodies satisfy the covariance property $Z_p(T_{\ast}\mu)=T(Z_p(\mu ))$ for every $T\in GL_n$,
where $(T_{\ast}\mu)(A)=\mu(T^{-1}(A))$. If $\mu$ is isotropic then $Z_2(\mu)=B_2^n$ by the normalization ${\rm Cov}(\mu)=I_n$. Moreover, for every $1\ls p<q$,
\begin{equation}\label{eq:zq-mu-inclusions}Z_p(\mu)\subseteq Z_q(\mu)\subseteq \frac{cq}{p}Z_p(\mu),
\end{equation}
where $c>0$ is an absolute constant. The asymptotic theory of centroid bodies was developed by Paouris in
\cite{Paouris-GAFA,Paouris-TAMS}.

If $K$ is a centered convex body and $\lambda_K$ denotes the uniform probability measure on $K$, we write
$Z_p(K):=Z_p(\lambda_K)$. Note that
\[Z_p(T(K))=T(Z_p(K)),\qquad T\in SL_n.\]
If $K$ is isotropic then the probability measure $\mu_K$, with density $L_K^n\mathds{1}_{\frac{K}{L_K}}$, is 
isotropic, and $Z_p(K)=L_KZ_p(\mu_K)$. For every $p\gr n$ and every $\xi\in S^{n-1}$ one has 
\begin{equation}\label{eq:stable}\|\langle \cdot,\xi\rangle\|_p\approx \|\langle \cdot,\xi\rangle\|_{\infty}\end{equation} 
(see \cite[Lemma~3.2.8/Proposition~5.1.2]{BGVV-book}) and hence $Z_p(K)\supseteq c_2Z_{\infty }(K)$, where $c_2>0$ is an absolute constant and 
$Z_{\infty}(K)={\rm conv}\{K,-K\}$. 

Recall that $\Psi_2(K)$ is the symmetric convex body with support function
\[ h_{\Psi_2(K)}(\xi)=\|\langle \cdot,\xi\rangle\|_{\psi_2},\qquad \xi\in S^{n-1}.\]
An important consequence for the present paper is the following representation of the $\psi_2$-norm:
\[\|\langle \cdot ,\xi\rangle\|_{L_{\psi_2}(K)}\approx\sup_{p\gr1}\frac{\|\langle \cdot ,\xi\rangle\|_{L_p(K)}}{\sqrt p}\approx
\sup_{1\ls p\ls n}\frac{\|\langle \cdot ,\xi\rangle\|_{L_p(K)}}{\sqrt p}\]
for every centered convex body $K$ and any $\xi\in S^{n-1}$, where the second equivalence is immediate from 
the fact that $p\mapsto \|\langle \cdot ,\xi\rangle\|_{L_p(K)}/\sqrt p$ is decreasing on $[n,\infty)$ by \eqref{eq:stable}. This implies
\begin{equation}\label{eq:convex-hull}\Psi_2(K )\approx{\rm conv}\left\{\frac{Z_p(K)}{\sqrt{p}}: 1\ls p\ls n\right\}.\end{equation} 
For a subspace $F \in G_{n,k}$, let $P_F:\mathbb{R}^n\to F$ denote the orthogonal projection. The marginal of a log-concave probability
measure $\mu$ with respect to $F$ is defined by
\begin{equation*}\pi_F(\mu)(A):=\mu(P_F^{-1}(A))
\end{equation*}for every Borel subset $A$ of $F$. Since linear images preserve log-concavity, $\pi_F(\mu)$ is again a log-concave probability
measure on $F$. Moreover, if $\mu$ is centered, then $\pi_F(\mu)$ is centered, and if $\mu$ is isotropic then
$\pi_F(\mu)$ is isotropic.

A basic observation of Paouris \cite{Paouris-GAFA} (see also \cite[Theorem~5.1.12]{BGVV-book}) is that projections commute with centroid bodies. Namely, 
for every $1\ls k\ls n$ and any $F \in G_{n,k}$, and every $p\gr 1$,
\begin{equation}\label{eq:marginal-Zq}
P_F(Z_p(\mu)) = Z_p(\pi_F(\mu)).
\end{equation}
For $q\neq 0$ and $q>-n$, define the $q$-th moment of the Euclidean norm by
\[I_q(\mu ):=\left( \int_{\mathbb{R}^n} |x|^q d\mu (x)\right)^{1/q}.\]
Combining the small-ball estimate \eqref{eq:optimal-small-ball} of 
Bizeul \cite{Bizeul-2025} with the standard representation of negative moments through small-ball probabilities, one obtains
\[I_{-(n-1)}(\mu)\approx I_2(\mu)\] 
for every isotropic log-concave probability measure $\mu$ on $\mathbb{R}^n$ (see \cite{Dafnis-Paouris-2010} and \cite[Theorem~6.9]{Giannopoulos-Pafis-Tziotziou-2025}).

The above equivalence of negative moments was used by Letwin and
Mikulincer \cite[Proposition~2.4]{Letwin-Mikulincer-2026} to obtain the following result.

\begin{proposition}\label{prop:2.4}
There exist absolute constants $c_0 \in (0,1/8]$ and $c,C > 0$ such that the following holds.
If $K \subset \R^n$ is an isotropic convex body, then for every $1 \ls p \ls 2c_0n$,
\[w_{-p}(Z_p(K))\approx w_{-2p}(Z_p(K))\approx \sqrt{p}\,L_K.\]
\end{proposition}

Here, for a convex body $L$ and $q\neq 0$, the $q$-th mean width is defined by
\begin{equation}\label{eq:q-mean-width}w_q(L)=\left(\int_{S^{n-1}}h_L(\xi)^qd\sigma(\xi)\right)^{1/q}.\end{equation}
In particular, this gives precise control of the negative moments of the support function of $Z_p(K)$.

The previous facts allow us to convert information on the family of centroid bodies
$Z_p(K)$ into information on the subgaussian body $\Psi_2(K)$.
The following construction provides the main mechanism for doing so.
Let $K$ be an isotropic convex body in $\mathbb{R}^n$ and fix $1 \ls p \ls c_0n$, where $c_0$ is the constant from Proposition~\ref{prop:2.4}. Define
\begin{equation}\label{eq:3.1}
A_p(K)=\{y \in \R^n : \|\langle \cdot,y\rangle\|_p \ls C_0\sqrt{np}\,L_K\}.
\end{equation}
Equivalently,
\[A_p:=A_p(K) = C_0\sqrt{np}\,L_K\,Z_p(K)^\circ.\]
In view of \eqref{eq:convex-hull}, one of the main new ideas of Letwin and Mikulincer \cite{Letwin-Mikulincer-2026} is to 
estimate the Gaussian measure of finite intersections of the sets $A_p$ from below. Using Proposition~\ref{prop:2.4} they obtain
the next lower bound for each individual $A_p$.

\begin{lemma}\label{lem:2.6}
For every $1 \ls p \ls 2c_0n$, if the constant $C_0$ in \eqref{eq:3.1} is chosen sufficiently large, then
\[\gamma_n(A_p) \gr e^{-C_1p},\]
where $C_1>0$ is an absolute constant.
\end{lemma}

Then, Letwin and Mikulincer combine Lemma~\ref{lem:2.6} with the Gaussian correlation inequality.

\begin{theorem}[Gaussian correlation inequality]\label{th:2.5}
Let $\gamma_n$ be the standard Gaussian measure on $\R^n$. If $A,B$ are symmetric convex Borel sets in $\mathbb{R}^n$, then
\[\gamma_n(A \cap B) \gr \gamma_n(A)\gamma_n(B).\]
By induction, for every finite family of symmetric convex Borel sets $A_1,\dots,A_m$,
\[\gamma_n\left(\bigcap_{j=1}^m A_j\right)\gr\prod_{j=1}^m \gamma_n(A_j).\]
\end{theorem}

The Gaussian correlation inequality was proved by Royen \cite{Royen-2014}; see also the exposition of Lata{\l}a and Matlak \cite{LM17}. 
We will use this inequality together with Lemma~\ref{lem:2.6} to obtain simultaneous control of the sets 
$A_p$, which is the key step in constructing large families of subgaussian directions.

%%%%%%%%%%%%%%%%%%%%%%%%%%%%%%%%%%%%%%%%%%%%%%%%%%%%%%%%%%%%%%%%%%%%%%%%%%%%%%%%%%%%%%%%%%%%%%%%%%%%%%%%%%%%%%%%%%%%%%%%%%%%%%%%%%%%%%%%%%%%%%%
\section{Geometry of the subgaussian body}\label{section:3}
%%%%%%%%%%%%%%%%%%%%%%%%%%%%%%%%%%%%%%%%%%%%%%%%%%%%%%%%%%%%%%%%%%%%%%%%%%%%%%%%%%%%%%%%%%%%%%%%%%%%%%%%%%%%%%%%%%%%%%%%%%%%%%%%%%%%%%%%%%%%%%%

In this section, we prove our main results concerning the geometry of the subgaussian body $\Psi_2(K)$ associated with a centered convex body $K$ in
$\mathbb{R}^n$. 

\medskip 

\noindent \textbf{\S~3.1. Volume estimates.} We begin with the proof of Theorem~\ref{th:vol-ratio-psi2-body}, which asserts that $\Psi_2(K)$ has uniformly bounded volume
ratio.

\begin{proof}[Proof of Theorem~$\ref{th:vol-ratio-psi2-body}$]
The subgaussian body $\Psi_2(K)$ satisfies the following affine invariance property: 
\[\Psi_2(T(K))=T(\Psi_2(K)),\qquad T\in SL_n.\]
Since the $L_p$-centroid bodies $Z_p(K)$ (and hence $Z_2(K)$) satisfy the same property, the volume ratio is affine invariant.
Consequently, we may assume that $K$ is isotropic.

Our starting point is \eqref{eq:convex-hull}:
\begin{equation*}\Psi_2(K )\approx {\rm conv}\left\{\frac{Z_p(K)}{\sqrt{p}}: 1\ls p\ls n\right\}.\end{equation*} 
By \eqref{eq:zq-mu-inclusions} we have $Z_{2p}(K )\approx Z_{p}(K)$ for all $p\gr 1$. This implies that the above
convex hull may be restricted, up to an absolute constant, to dyadic values
of $p$. Hence,
\begin{equation*}\Psi_2(K)\approx {\rm conv} \left\{\frac{Z_{2^k}(K )}{\sqrt{2^k}}: k=0,1,\ldots ,\lfloor\log_{2}n\rfloor \right\}. \end{equation*}
In fact, using the inclusion $Z_n(K)\subseteq \frac{c}{c_0}Z_{c_0n}(K)$, where $c_0$ is the absolute constant in Proposition~\ref{prop:2.4}, we see that
\begin{equation*}\Psi_2(K)\approx {\rm conv} \left\{\frac{Z_{2^k}(K )}{\sqrt{2^k}}: k=0,1,\ldots ,k_0 \right\}. \end{equation*}
where $k_0=\lceil \log_2(c_0n)\rceil$. Taking polars, we obtain
\[\Psi_2(K )^{\circ}\approx \bigcap_{k=0}^{k_0}\sqrt{2^k}Z_{2^k}(K)^{\circ}.\]
By the definition of the sets $A_p$ in \eqref{eq:3.1},
\[\sqrt{p}\,Z_p(K)^\circ=\frac{1}{C_0\sqrt n\,L_K}A_p .\]
Therefore,
\[\Psi_2(K)^\circ\approx\frac{1}{C_0\sqrt n\,L_K}\bigcap_{k=0}^{k_0}A_{2^k}.\]
Set
\[A=\bigcap_{k=0}^{k_0}A_{2^k}.\]
Each $A_{2^k}$ is a symmetric convex set. Therefore, the Gaussian correlation inequality (Theorem~\ref{th:2.5}) gives 
\[\gamma_n(A)=\gamma_n\left(\bigcap_{k=0}^{k_0 } A_{2^k}\right)\gr\prod_{k=0}^{k_0} \gamma_n(A_{2^k}).\]
Applying Lemma~\ref{lem:2.6}, we obtain
\[\gamma_n(A)\gr\prod_{k=0}^{k_0} e^{-C_12^k}=
\exp\left(-C_1\sum_{k=0}^{k_0}2^k\right)\gr e^{-C_2n}.\]
Since
\[\gamma_n(A)=(2\pi)^{-n/2}\int_Ae^{-|x|^2/2}dx\]
and $e^{-|x|^2/2}\ls 1$, we have
\[{\rm vol}_n(A)\gr \int_Ae^{-|x|^2/2}dx= (2\pi)^{n/2}\gamma_n(A) \gr e^{-C_3n}.\]
Consequently,
\[{\rm vol}_n(\Psi_2(K )^{\circ})\gr \left(\frac{c}{\sqrt{n}L_K}\right)^n.\]
By the Blaschke--Santal\'{o} inequality,
\[{\rm vol}_n(\Psi_2(K)){\rm vol}_n(\Psi_2(K )^{\circ})\ls\omega_n^2,\]
together with $\omega_n^{1/n}\approx\frac{1}{\sqrt{n}}$ and the boundedness of the isotropic constant $L_K\ls C$, we conclude that
\[{\rm vrad}(\Psi_2(K))\ls C.\]
This proves the theorem.
\end{proof}

For a centered log-concave probability measure $\mu$ on $\mathbb{R}^n$, we define the symmetric convex body $\Psi_{2}(\mu)$
by 
\begin{equation*} h_{\Psi_{2}(\mu)}(v ):= \sup_{1\ls q\ls n}\frac{\|\langle \cdot, v \rangle \|_{L_q(\mu)}}{\sqrt{q}}, \qquad v\in\mathbb{R}^n.\end{equation*} 
If $K$ is isotropic then the probability measure $\mu_K$, with density $L_K^n\mathds{1}_{\frac{K}{L_K}}$, satisfies $\Psi_2(K)\approx L_K\Psi_2(\mu_K)\approx \Psi_2(\mu_K)$.

We shall use an observation from \cite{GPV11}. A proof of the next proposition, in this exact formulation, can be found in \cite[Proposition~8.1.7]{BGVV-book}
(we also take into account the fact that all isotropic constants are now known to be of order $1$ up to an absolute constant).

\begin{proposition}\label{prop:Zq-control}
Let $\mu$ be an isotropic log-concave probability measure on $\mathbb{R}^n$. Then there exists a centered convex body $T=T_\mu$ in $\mathbb{R}^n$ with
the following property: For every $1\ls q\ls n$,
\begin{equation*}c_1Z_q(T)\subseteq Z_q(\mu)+\sqrt{q}B_2^n\subseteq c_2 Z_q(T),\end{equation*} 
where $c_1,c_2>0$ are absolute constants. In particular, $Z_q(\mu)\subseteq c_2Z_q(T)$ for every $1\ls q\ls n$.
\end{proposition} 

The convex body $T$ with these properties can be taken to be
\[T=\overline{K_{n+1}}(\mu\ast\gamma_n),\]
where $\gamma_n$ denotes the standard Gaussian measure on $\mathbb{R}^n$.
See~\cite[Chapter~2]{BGVV-book} for the construction of
K.~Ball's bodies $K_p(\mu)$.

\medskip

We now extend the volume estimate for $\Psi_2$ from convex bodies to
log-concave measures.

\begin{theorem}\label{th:vol-ratio-psi2-body-measure}
Let $\mu$ be a centered log-concave probability measure on $\mathbb{R}^n$. Then,
\begin{equation*}\left(\frac{{\rm vol}_n(\Psi_{2}(\mu))}{{\rm vol}_n(Z_{2}(\mu))}\right)^{1/n} \ls C,\end{equation*} 
where $C>0$ is an absolute constant. 
\end{theorem}

\begin{proof}By affine invariance, we may assume that $\mu$ is isotropic. Let
$T=T_\mu$ be the convex body given by Proposition~\ref{prop:Zq-control}.
From the definition of $\Psi_2$ and Proposition~\ref{prop:Zq-control}, we have
\begin{equation}\label{eq:body-measure}\Psi_2(\mu)\subseteq c_2\Psi_2(T).\end{equation} 
Since $\mu$ is isotropic, we have $Z_2(\mu)=B_2^n$. Moreover, ${\rm vol}_n(Z_2(T))=\omega_nL_T^n$ and hence,
\[{\rm vol}_n(Z_2(T))={\rm vol}_n(Z_2(\mu))L_T^n.\]
Combining the above and applying Theorem~\ref{th:vol-ratio-psi2-body} to the centered convex body $T$, we obtain
\[\left(\frac{{\rm vol}_n(\Psi_{2}(\mu))}{{\rm vol}_n(Z_{2}(\mu))}\right)^{1/n}
\ls c_2L_T\left(\frac{{\rm vol}_n(\Psi_2(T))}{{\rm vol}_n(Z_2(T))}\right)^{1/n}
\ls c_2CL_T.\]
Since $L_T$ is bounded by an absolute constant, this completes the proof. 
\end{proof}

\medskip 

We now prove Theorem~\ref{th:vk-psi2-body}, which gives an upper bound for
the volume of projections of $\Psi_2(K)$ in the isotropic case.

\begin{proof}[Proof of Theorem~$\ref{th:vk-psi2-body}$]
Let $F\in G_{n,m}$. By the projection property of centroid bodies, \eqref{eq:marginal-Zq}, we have
\begin{align*}P_F(\Psi_2(K))&\approx {\rm conv}\left\{\frac{P_F(Z_p(K))}{\sqrt p}:p=2^k,\ k=0,1,\ldots,\lfloor\log_2n\rfloor\right\}\\
&={\rm conv}\left\{\frac{L_KZ_p(\pi_F(\mu_K))}{\sqrt p}:p=2^k,\ k=0,1,\ldots,\lfloor\log_2n\rfloor\right\}.
\end{align*}
If $p>m$, the moment comparison inequality \eqref{eq:zq-mu-inclusions} in dimension $m$ gives
\[Z_p(\pi_F(\mu_K))\subseteq c\frac{p}{m}Z_m(\pi_F(\mu_K)).\]
Hence,
\[\frac{Z_p(\pi_F(\mu_K))}{\sqrt p}\subseteq c\sqrt{\frac pm}\frac{Z_m(\pi_F(\mu_K))}{\sqrt m}.\]
Since $p\ls n$, we obtain
\[\frac{Z_p(\pi_F(\mu_K))}{\sqrt p}\subseteq c\sqrt{\frac nm}\frac{Z_m(\pi_F(\mu_K))}{\sqrt m}.\]
Therefore,
\begin{align*}
P_F(\Psi_2(K))&\subseteq c\sqrt{\frac nm}\,L_K\,{\rm conv}\left\{\frac{Z_p(\pi_F(\mu_K))}{\sqrt p}: p=2^k,\ k=0,1,\ldots,\lfloor\log_2m\rfloor\right\}\\
&\subseteq c\sqrt{\frac nm}\,L_K\Psi_2(\pi_F(\mu_K)).
\end{align*}
The marginal $\pi_F(\mu_K)$ is isotropic on $F$. Applying Theorem~\ref{th:vol-ratio-psi2-body-measure} in dimension $m$ gives
\[{\rm vrad}(\Psi_2(\pi_F(\mu_K)))\ls C.\]
Consequently,
\[{\rm vrad}(P_F(\Psi_2(K)))\ls C\sqrt{\frac nm},\]
which proves the theorem.
\end{proof}

Choosing $m=1$ in Theorem~\ref{th:vk-psi2-body} we see that ${\rm vol}_1(P_F(\Psi_2(K)))\ls 2C\sqrt{n}$ for every $F\in G_{n,1}$. Since
$\Psi_2(K)$ is symmetric, this is equivalent to $h_{\Psi_2(K)}(v)\ls C\sqrt{n}$ for every $v\in S^{n-1}$. Hence, we obtain the next fact.

\begin{corollary}\label{cor:radius}
Let $K$ be an isotropic convex body in $\mathbb{R}^n$. Then $R(\Psi_2(K))\ls C\sqrt{n}$, where $C>0$ is an absolute constant.
\end{corollary}

The same fact can be also seen as follows: since $K$ is isotropic, standard estimates for isotropic convex bodies (that can be found e.g.~in \cite[Chapter~3]{BGVV-book}) give
$R(K)\ls c_1 nL_K$ and $\|\langle\cdot,v\rangle\|_{\psi_1}\ls c_2L_K$ for all $v\in S^{n-1}$, where
\[\|\langle\cdot,v\rangle\|_{\psi_1}=\inf\left\{t>0:\mathbb E\exp\!\left(|\langle \cdot,v\rangle |/t\right)\ls 2\right\}.\]
Consequently,
\[\|\langle\cdot,v\rangle\|_{\psi_2}\ls C\sqrt{\|\langle\cdot,v\rangle\|_{\psi_1}\|\langle\cdot,v\rangle\|_\infty}
\ls C_1\sqrt{n}L_K \ls C\sqrt{n}.\]

\medskip 

\noindent \textbf{\S~3.2. Construction of subgaussian bases.}
We next use Theorem~\ref{th:vk-psi2-body} to construct an orthonormal basis
consisting of directions with controlled subgaussian norms, following the
approach of~\cite{Letwin-Mikulincer-2026}.

\begin{proof}[Proof of Theorem~$\ref{th:onb-psi2}$]
Since $Z_2(K)=L_KB_2^n\approx B_2^n$, by Theorem~\ref{th:vol-ratio-psi2-body} we have that 
$r:={\rm vrad}(\Psi_2(K))\ls C$ for an absolute constant $C>0$. Note that 
\begin{equation}\label{eq:min}\min\{h_{\Psi_2(K)}(v):v\in S^{n-1}\}\ls r.\end{equation} 
Otherwise, there exists $s>r$ such that
$h_{\Psi_2(K)}(v)\gr s=h_{sB_2^n}(v)$ for all $v\in S^{n-1}$, which implies $\Psi_2(K)\supseteq sB_2^n$ and hence
${\rm vrad}(\Psi_2(K))\gr s>r$, a contradiction. 
From \eqref{eq:min} we see that there exists $v_1\in S^{n-1}$ such that
\[\|\langle\cdot,v_1\rangle\|_{\psi_2}\ls r={\rm vrad}(\Psi_2(K))\ls C .\]

We proceed inductively. Assume that $v_1,\ldots,v_{k-1}$ have been chosen so that they form an orthonormal sequence and satisfy
\[\|\langle\cdot,v_i\rangle\|_{\psi_2}\ls C\sqrt{\frac{n}{n-i+1}},\qquad i=1,\ldots,k-1 .\]
Let $F_k=({\rm span}\{v_1,\ldots,v_{k-1}\})^\perp $. Then $F_k\in G_{n,n-k+1}$. By Theorem~\ref{th:vk-psi2-body},
\[{\rm vrad}(P_{F_k}(\Psi_2(K)))\ls C\sqrt{\frac{n}{n-k+1}} .\]
Since $P_{F_k}(\Psi_2(K))$ is a convex body in $F_k$, there exists a vector $v_k\in S^{n-1}\cap F_k$ such that
\[h_{P_{F_k}(\Psi_2(K))}(v_k)\ls C\sqrt{\frac{n}{n-k+1}} .\]
Because $v_k\in F_k$,
\[h_{P_{F_k}(\Psi_2(K))}(v_k)=h_{\Psi_2(K)}(v_k),\]
and therefore
\[\|\langle\cdot,v_k\rangle\|_{\psi_2}\ls C\sqrt{\frac{n}{n-k+1}} .\]
By construction, the vectors remain orthonormal. Iterating the argument gives an orthonormal basis
$\{v_1,\ldots,v_n\}$ with the desired estimates.
\end{proof}

\medskip 

\noindent \textbf{\S~3.3. Mean width and mean norm.} Recall from \eqref{eq:Bizeul-M} that $w(K)\ls C\sqrt{n\ln(en)}$ for every isotropic
convex body $K$ in $\mathbb{R}^n$. We begin with a version of this result for isotropic log-concave probability measures.

\begin{lemma}\label{lem:measure-bizeul}Let $\nu$ be an isotropic log-concave probability measure on $\mathbb{R}^m$. Then
\begin{equation}\label{eq:width-measure}w(Z_m(\nu))\ls C\sqrt{m\ln(em)}.\end{equation}
\end{lemma}

\begin{proof}Let $f_{\nu}$ denote the density of $\nu$. We shall use a number of known facts abour Ball's bodies $K_p(\nu)$ associated with $\nu$. 
Recall that, for any $p>0$, the convex body $K_p(\nu)$ has radial function
$$\rho_{K_p(\nu)}(x)=\left (\frac{1}{f_{\nu}(0)}\int_0^{\infty}pr^{p-1}f_{\nu}(rx)\,dr\right )^{1/p}$$ for $x\neq 0$.
We refer the reader to \cite{BGVV-book} for more details about the family $\{K_p(\nu)\}_{p>0}$. From \cite[Proposition~2.5.8]{BGVV-book} we have 
$$f_{\nu}(0) \vol_m(K_{m+1}(\nu)) \approx 1.$$
Then, \cite[Proposition~2.5.12]{BGVV-book} shows that if we set 
$$a:=\vol_m(K_{m+1}(\nu))^{-1/m}\approx f_{\nu}(0)^{1/m}\approx L_{\nu}\approx 1$$ 
then $M=aK_{m+1}(\nu)$ has volume one and is almost isotropic. 
This means (see \cite[Definition~2.5.11]{BGVV-book}) that if $T\in SL_m$ is chosen so that $T(M)$ is isotropic then $T(B_2^m)\approx B_2^m$.
By Bizeul's estimate,
\begin{equation}\label{eq:kball-width}w(K_{m+1}(\nu))\approx w(T(M))\ls C\sqrt{m\ln(em)}.\end{equation}
On the other hand, \cite[Theorem~5.1.7]{BGVV-book} and the fact that $f_{\nu}(0)^{1/m}\approx L_{\nu}\approx 1$ imply
$$Z_m(\nu)\approx Z_m(M).$$
Since $Z_m(M)\subseteq \operatorname{conv}\{M,-M\}$, we conclude that $w(Z_m(\nu))\ls C_1\,w(K_{m+1}(\nu))$, and the lemma follows from \eqref{eq:kball-width}.
\end{proof}

Using Lemma~\ref{lem:measure-bizeul} we obtain the following consequence of Bizeul's sharp bound. 

\begin{lemma}\label{lem:Zq-width}Let $K$ be an isotropic convex body in $\mathbb{R}^n$. For every $1\ls q\ls n$,
$$w(Z_q(K))\ls C\sqrt{q\ln(eq)}.$$
\end{lemma}

\begin{proof}Without loss of generality, we may assume that $q$ is an integer. Let $\mu$ be the isotropic measure with density $L_K^n\mathds{1}_{\frac{K}{L_K}}$.
For any $F\in G_{n,q}$ we have that $\pi_F(\mu)$ is isotropic on the $q$-dimensional space $F$. Since, $P_F(Z_q(\mu))=Z_q(\pi_F(\mu))$, applying \eqref{eq:width-measure}
for $\pi_F(\mu)$ we get
$$w_F(P_F(Z_q(\mu)))=w_F(Z_q(\pi_F(\mu)))\ls C\sqrt{q\ln(eq)}.$$
Now, we use the identity
\begin{align*}w(A) &=\int_{G_{n,q}}\int_{S_F}h_A(\xi)\,d\sigma_F(\xi)\,d\nu_{n,q}(F)=\int_{G_{n,q}}\int_{S_F}h_{P_F(A)}(\xi)\,d\sigma_F(\xi)\,d\nu_{n,q}(F)\\
&=\int_{G_{n,q}}w_F(P_F(A))\,d\nu_{n,q}(F)\end{align*}
to write
$$w(Z_q(\mu))=\int_{G_{n,q}}w_F(P_F(Z_q(\mu)))\,d\nu_{n,q}(F)\ls C\sqrt{q\ln(eq)}$$
completing the proof.
\end{proof}

\begin{theorem}\label{th:1.3}Let $K$ be an isotropic convex body in ${\mathbb{R}}^n$. Then,
\[w(\Psi_2(K))\ls C\sqrt{\ln(en)},\]
where $C>0$ is an absolute constant.
\end{theorem}

\begin{proof}We start from the formula
\begin{equation*}\Psi_2(K)\approx {\rm conv} \left\{\frac{Z_{2^k}(K )}{\sqrt{2^k}}: k=0,1,\ldots ,\lfloor\log_{2}n\rfloor \right\}. \end{equation*}
This implies that for any $s\gr 1$ we have
\[h_{\Psi_2(K)}(v)\approx \max_{1\ls k\ls \lfloor\log_2n\rfloor}\frac{h_{Z_{2^k}}(v)}{\sqrt{2^k}}
\ls c_5\left(\sum_{k=0}^{\lfloor\log_2n\rfloor}\frac{h^s_{Z_{2^k}}(v)}{(\sqrt{2^k})^s}\right)^{1/s}.\]
We set $q_k=2^k$. Integrating on $S^{n-1}$ and using H\"{o}lder's inequality, we obtain
\begin{equation}\label{eq:main-w}w(\Psi_2(K))\ls c_5\left(\sum_{k=0}^{\lfloor\log_2n\rfloor}\frac{1}{(\sqrt{q_k})^s}\int_{S^{n-1}}h^s_{Z_{q_k}}(v)\,d\sigma(v)\right)^{1/s}
=c_5\left(\sum_{k=0}^{\lfloor\log_2n\rfloor}\left(\frac{w_s(Z_{q_k}(K)}{\sqrt{q_k}}\right)^s\right)^{1/s},\end{equation}
where the $s$-mean width $w_s(C)$ of a symmetric convex body $C$ in $\mathbb{R}^n$ is defined in \eqref{eq:q-mean-width}.
These parameters were studied  by Litvak, V.~Milman and Schechtman who showed in \cite{Litvak-VMilman-Schechtman-1998} that
\begin{equation}\label{eq:LMS}w_s(C)\approx \max\bigg\{w(C),\frac{R(C)\sqrt{s}}{\sqrt{n}}\bigg\}
\end{equation} for all $1\ls s\ls n$.

We choose $s=\ln(en)$. We shall use the fact that for every $k$ we have $R(Z_{q_k}(K))\ls c_6q_k$ (see \cite[Chapter~5]{BGVV-book}).
From \eqref{eq:LMS} we get
\begin{equation}\label{eq:LMS-Zq}\frac{w_s(Z_{q_k}(K))}{\sqrt{q_k}}\ls c_8\max\bigg\{ \frac{w(Z_{q_k}(K))}{\sqrt{q_k}},\frac{\sqrt{q_ks}}{\sqrt{n}}\bigg\}
\ls c_9\max\bigg\{ \sqrt{\ln(eq_k)},\sqrt{s}\bigg\}\ls c_{10}\sqrt{\ln(en)}\end{equation}
for every $k=0,1,\ldots ,\lfloor\log_2n\rfloor$. 
Since there are only $O(\ln(en))$
dyadic indices and $s\approx \ln(en)$, the factor $\ln^{1/s}(en)$ is bounded by an absolute
constant. Going back to \eqref{eq:main-w} we conclude the proof.
\end{proof}

We also observe that $\Psi_2(K)$ has bounded  $M$-parameter. From $\Psi_2(K)\supseteq c\,Z_2(K)$ we readily see that, for every isotropic convex body $K$ in $\mathbb{R}^n$,
\[M(\Psi_2(K))=\int_{S^{n-1}}\|\xi\|_{\Psi_2(K)}\,d\sigma(\xi)\ls C.\]
Combining this estimate with Theorem~\ref{th:1.3}, we obtain
\begin{equation}\label{eq:ell-psi}M(\Psi_2(K))\,w(\Psi_2(K))\ls C\sqrt{\ln(en)},\end{equation}
which means that the subgaussian body of an isotropic convex body is essentially in the $\ell$-position.

\begin{remark}\rm Similar arguments extend Lemma~\ref{lem:Zq-width} and Theorem~\ref{th:1.3} to isotropic log-concave probability measures.
\end{remark}

\begin{proposition}\label{prop:width-measures}Let $\mu$ be an isotropic log-concave probability measure on $\mathbb{R}^n$. Then
\begin{equation}\label{eq:Zq-width-measure}w(Z_q(\mu))\ls C\sqrt{q\ln(eq)},\qquad 1\ls q\ls n\end{equation}
and 
\[w(\Psi_2(\mu))\ls C\sqrt{\ln(en)},\]
where $C>0$ is an absolute constant.
\end{proposition}

\medskip 

\noindent \textbf{\S~3.4. Large sets of subgaussian directions.}
In this subsection we study the distribution of the $\psi_2$-norm over directions. 
Let $K$ be an isotropic convex body in $\mathbb{R}^n$. Brazitikos and Hioni~\cite{Brazitikos-Hioni-2015} obtained
logarithmic estimates for the $\psi_2$-norm of most directions. More precisely, they proved that, for every $a>1$, there exists a set
$V_a\subseteq S^{n-1}$ with $\sigma(V_a)\gr 1-n^{-a}$ such that
\[\|\langle\cdot,v\rangle\|_{\psi_2}\ls C\ln^{3/2}(en)\max\{\sqrt{\ln(en)},\sqrt a\}\]
for every $v\in V_a$.

In view of the bound $w(\Psi_2(K))\ls C\sqrt{\ln(en)}$ of Theorem~\ref{th:1.3}, the next theorem yields that a random orthonormal basis of $\mathbb{R}^n$ consists of
$\psi_2$-directions with better behavior. Theorem~\ref{th:partial-onb} is an immediate consequence of Proposition~\ref{prop:spherical} below, applied e.g. with $\delta =1/2$.

\begin{theorem}\label{th:spherical}Let $K$ be an isotropic convex body in $\mathbb{R}^n$. For every $t>0$ there exists a
subset $V_t\subseteq S^{n-1}$ with
\[\sigma(V_t)\gr 1-\exp(-t^2/2)\]
such that, for every $v\in V_t$,
\begin{equation}\label{eq:onb-bound}
\|\langle\cdot,v\rangle\|_{\psi_2}=h_{\Psi_2(K)}(v)\ls w(\Psi_2(K))+Ct\ls C\big(\sqrt{\ln(en)}+t\big),
\end{equation}
where $c,C>0$ are absolute constants. 
\end{theorem}

The proof of Theorem~\ref{th:spherical} relies on the classical deviation inequality for Lipschitz functions $f:S^{n-1}\to {\mathbb R}$.
(see~\cite[Proposition~5.20]{Aubrun-Szarek-book} for the exact formulation below), which is a consequence of the spherical isoperimetric inequality.

\begin{theorem}\label{th:sphere-expectation} Let $f:S^{n-1}\to {\mathbb R}$ be a Lipschitz continuous function with constant $b$. Then, for every $s>0$,
\begin{equation*}\sigma\left ( \{v\in S^{n-1}:f(v)\gr {\mathbb E}(f)+bs\}\right )\ls \exp (-s^2n/2),\end{equation*}
where ${\mathbb E}(f)$ is the expectation of $f$.
\end{theorem}

\begin{proof}[Proof of Theorem~$\ref{th:spherical}$]Corollary~\ref{cor:radius} implies that $f(v)=\|\langle \cdot, v\rangle\|_{\psi_2}$ 
is Lipschitz continuous on $S^{n-1}$ with constant $C\sqrt{n}$.
Note that $\mathbb{E}(f)=w(\Psi_2(K))$. Therefore, Theorem~\ref{th:sphere-expectation} implies that
\begin{equation*}\sigma\left ( \{v\in S^{n-1}:\|\langle\cdot,v\rangle\|_{\psi_2}\gr w(\Psi_2(K))+Ct
\}\right )\ls \exp (-t^2/2)\end{equation*}
for every $t>0$.

Define
\[V_t=\big\{v\in S^{n-1}:\|\langle \cdot, v\rangle\|_{\psi_2}\ls w(\Psi_2(K))+Ct\big\}.\]
Then, $\sigma(V_t)\gr 1-\exp(-t^2/2)$, which is precisely \eqref{eq:onb-bound}.
\end{proof}

The next proposition shows that a random orthonormal basis of $\mathbb{R}^n$ consists entirely of subgaussian directions,
with subgaussian constants growing at most logarithmically in the dimension.

\begin{proposition}\label{prop:spherical}
There exists an absolute constant $C_1>0$ such that, for every isotropic convex body $K$ in $\mathbb{R}^n$, 
$$\mathbb{E} \max_{1\ls i\ls n} \|\langle \cdot,U(e_i)\rangle \|_{\psi_2}  \ls C_1\sqrt{\ln (en)}.$$
Moreover, for any $\delta\in (0,1)$, a random orthonormal basis $\{v_1,\ldots ,v_n\}$ of $\mathbb{R}^n$ satisfies with probability at least $1-\delta$ 
\begin{equation}\label{eq:random-basis}
\|\langle\cdot,v_i\rangle\|_{\psi_2}\ls w(\Psi_2(K))+C_1\sqrt{\ln(2n/\delta)}\ls C_2\sqrt{\ln(2n/\delta)},\qquad i=1,\ldots ,n.
\end{equation}
where $C_2>0$ is an absolute constant. 
\end{proposition}

\begin{proof}
Let $U \in O(n)$ and define $f_i(U)=\|\langle \cdot,U(e_i)\rangle \|_{\psi_2}$.  Then 
\begin{align*}
\mathbb{E}\max_{1\ls i\ls n} \|\langle \cdot,U(e_i)\rangle \|_{\psi_2}&=\int_0^{\infty} \nu(\{U \in O(n) : \max_{1\ls i\ls n}f_i(U)\geqslant t\})\, dt \\& \ls t_0+n\int_{t_0}^{\infty} \nu(\{U \in O(n) :f_1(U)\geqslant t\}) \, dt \\& = t_0+n\int_{t_0}^{\infty} \sigma(\{v\in S^{n-1} : h_{\Psi_2 (K)}(v)\geqslant t\}) \, dt,
\end{align*}
where $t_0$ will be chosen later. Assume that $$t_0\geqslant 2C\sqrt{2\ln(en)}+ w(\Psi_2 (K)),$$ where  $C>0$ is the absolute constant in Theorem~\ref{th:spherical}. By the change of variables $t=w(\Psi_2(K))+Cs$
we obtain  
$$\int_{t_0}^{\infty} \sigma(\{v\in S^{n-1} : h_{\Psi_2 (K)}(v)\geqslant t\}) \, dt= C\int_{\frac{1}{C}(t_0-w(\Psi_2(K)))}^{\infty} \sigma(\{v\in S^{n-1} : h_{\Psi_2 (K)}(v)\geqslant w(\Psi_2(K))+Cs\}) \, ds.$$
 Then, for every $s>0$ there exists $V_s\subset S^{n-1}$ such that $\sigma(V_s)\geqslant 1-\exp{(-s^2/2)}$ 
and for every $v \in V_s$ we have that $h_{\Psi_2(K)}(v)\ls w(\Psi_2(K))+Cs$. Thus 
$$\sigma(\{v\in S^{n-1} : h_{\Psi_2 (K)}(v)\geqslant w(\Psi_2(K))+Cs\}) \ls \sigma(S^{n-1}\setminus V_s)\ls \exp{(-s^2/2)}$$
Using the change of variables $t=2s\sqrt{2\ln (en)}$ we obtain
\begin{align*}
\int_{\frac{1}{C}(t_0-w(\Psi_2(K)))}^{\infty}\exp{(-t^2/2)}\,dt &\ls \int_{2\sqrt{2 \ln(en)}}^{\infty}\exp{(-t^2/2)}\,dt\\
&= 2\sqrt{2\ln (en)}\int_1^{\infty} \exp{(-4s^2\ln (en))} \, ds\\
&\ls 2\sqrt{2\ln (en)}\int_1^{\infty} \exp{(-4s\ln (en))}\, ds \\
& \ls \frac{2}{e^3}\sqrt{2\ln (en)}\, n^{-3}\int_1^{\infty} e^{-s} \, ds = \frac{2}{e^4}n^{-3}\sqrt{2\ln (en)},
\end{align*}
where we used the inequality $\exp{(-4s\ln (en))}\ls e^{-s}\exp{(-3\ln (en))}$ which holds for every $s\geqslant 1$.  Finally, since $w(\Psi_2(K))\ls C\sqrt{\ln (en)}$, we may choose $t_0=C_0\sqrt{\ln (en)}$ where $C_0>0$ is a sufficiently large absolute constant.
Therefore, combining the above estimates, we obtain
$$\mathbb{E}\max_{1\ls i\ls n} \|\langle\cdot,U(e_i)\rangle\|_{\psi_2} \ls C_1\sqrt{\ln (en)}.$$    

For the second claim, let $t>0$ and consider the set $V_t$ from Theorem~\ref{th:spherical}. Fix an orthonormal basis $\{e_1,\ldots ,e_n\}$ and observe that, for every $i=1,\ldots ,n$,
\[\nu(\{U\in O(n):U(e_i)\notin V_t\})\ls \exp(-t^2/2).\]
If $n\exp(-t^2/2)<\delta$ then $\{U(e_1),\ldots ,U(e_n)\}\subset V_t$ with probability greater than $1-\delta$. This shows that
\[\nu(\{U\in O(n):\max_{1\ls i\ls n}\|\langle \cdot,U(e_i)\rangle \|_{\psi_2}\ls w(\Psi_2(K))+(C/\sqrt{c})\sqrt{\ln(2n/\delta)}\})\gr 1-\delta.\]
In other words, a random orthonormal basis $\{v_1,\ldots ,v_n\}$ of $\mathbb{R}^n$ satisfies \eqref{eq:random-basis} with probability at least $1-\delta$.
\end{proof}

We can also show that if $K$ is isotropic then, for any $1\ls k\ls n-1$, a random $F\in G_{n,k}$ is ``subgaussian with constant $\sqrt{k}$ for $K$" in the sense
that $\|\langle\cdot ,v\rangle\|_{\psi_2}\ls C\max\{w(\Psi_2(K)),\sqrt{k}\}$ {\it for all} $v\in S^{n-1}\cap F$. 

\begin{proposition}\label{prop:1.2}Let $K$ be an isotropic convex body in $\mathbb{R}^n$. For every $1\ls k\ls n-1$ there exists a subset $\Gamma_k\subseteq G_{n,k}$ with
$\nu_{n,k}(\Gamma_k)\gr 1-e^{-ck}$ such that for every $F\in\Gamma_k$ and every $v\in S^{n-1}\cap F$,
\[\|\langle\cdot ,v\rangle\|_{\psi_2}\ls C\max\{w(\Psi_2(K)),\sqrt{k}\}\ls C_1\max\{\sqrt{\ln(en)},\sqrt{k}\}.\]
\end{proposition}

The proof of Proposition~\ref{prop:1.2} relies on a result of V.~Milman~\cite{VMilman-1990}; see also
\cite[Proposition~5.7.1]{AGA-book}, which gives a bound for the diameter of
random low-dimensional projections of a symmetric convex body.

\begin{proposition}\label{prop:diam-rdm-proj}
Let $K$ be a symmetric convex body in $\mathbb{R}^n$ satisfying $K\subseteq RB_2^n$.
For every $1\ls k\ls n-1$, there exists a subset $\Gamma_k\subseteq G_{n,k}$ with $\nu_{n,k}(\Gamma_k)\gr 1-e^{-ck}$
such that for every $F\in\Gamma_k$,
\[P_F(K)\subseteq C\max\left\{w(K),R\sqrt{k/n}\right\}B_2^n\cap F ,\]
where $c>0$ and $C>0$ are absolute constants.
\end{proposition}

\begin{proof}[Proof of Proposition~$\ref{prop:1.2}$]
Recall from the proof of Theorem~\ref{th:spherical} that
\[R(\Psi_2(K)) \ls C\sqrt n.\] 
Given $k$, we apply Proposition~\ref{prop:diam-rdm-proj} to the symmetric convex body
$\Psi_2(K)$ to find a subset $\Gamma_k\subseteq G_{n,k}$ with
\[\nu_{n,k}(\Gamma_k)\gr 1-e^{-ck}\]
such that for every $F\in\Gamma_k$,
\[P_F(\Psi_2(K))\subseteq C\max\left\{w(\Psi_2(K)),R(\Psi_2(K))\sqrt{k/n}\right\}B_2^n\cap F .\]
Using the estimate for the radius of $\Psi_2(K)$, we have
\[R(\Psi_2(K))\sqrt{k/n}\ls C_1\sqrt{k}.\]
Therefore,
\[P_F(\Psi_2(K))\subseteq C\max\{w(\Psi_2(K)),\sqrt{k}\}B_2^n\cap F\]
for every $F\in\Gamma_k$. Equivalently, for every $F\in\Gamma_k$ and any $v\in S^{n-1}\cap F$,
\[h_{P_F(\Psi_2(K))}(v)\ls C\max\{w(\Psi_2(K)),\sqrt{k}\}.\]
Since $h_{P_F(\Psi_2(K))}(v)=h_{\Psi_2(K)}(v)=\|\langle\cdot ,v\rangle\|_{\psi_2}$, the result follows.
\end{proof}

\medskip 

\noindent \textbf{\S~3.5. The case of unconditional bodies.}
Recall that a convex body $K$ in $\mathbb{R}^n$ is called unconditional if $x=\sum_{i=1}^n x_i e_i\in K$ implies that
$\sum_{i=1}^n\varepsilon_i x_i e_i\in K$ for every choice of signs $\varepsilon_i\in\{-1,1\}$, where
$\{e_1,\ldots,e_n\}$ is the standard orthonormal basis of $\mathbb{R}^n$.
We conclude this article by showing that unconditionality yields a particularly strong form of the previous results: one can construct an explicit orthonormal basis all of whose vectors are uniformly subgaussian directions.

We begin with the model example of the unit ball of $\ell_1^n$ equipped with the uniform probability measure $\mu_n$, whose density is
\[\frac{d\mu_n}{dx}=\frac{n!}{2^n}\mathds{1}_{B_1^n}.\]
A direct computation (see \cite{Bobkov-Nazarov-2003}) shows that, for every $v\in\mathbb{R}^n$,
\begin{equation}\label{eq:mu-n}
c_1\|v\|_\infty\ls\sqrt n\,\|\langle\cdot,v\rangle\|_{\psi_2(\mu_n)}\ls c_2\|v\|_\infty,
\end{equation}
where one may take $c_1=1/\sqrt6$ and $c_2=2\sqrt2$.

Bobkov and Nazarov \cite{Bobkov-Nazarov-2003} proved a comparison theorem implying, in particular, that the upper estimate in \eqref{eq:mu-n} extends to 
arbitrary isotropic unconditional convex bodies.

\begin{theorem}[Bobkov--Nazarov]\label{th:unco-psi-2}
Let $K$ be an isotropic unconditional convex body in $\mathbb{R}^n$. Then, for every $v\in\mathbb{R}^n$,
\[\|\langle\cdot,v\rangle\|_{\psi_2}\ls C\sqrt n\,\|v\|_\infty.\]
\end{theorem}

The next elementary lemma provides a convenient orthonormal basis adapted to the $\ell_\infty$-norm.

\begin{lemma}\label{lem:fourier-basis}
For every $n\ge1$, there exists an orthonormal basis $\{v_0,\ldots,v_{n-1}\}$ of $\mathbb{R}^n$ such that
\[\|v_i\|_\infty\ls\frac{\sqrt2}{\sqrt n},\qquad i=0,\ldots,n-1.\]
\end{lemma}

\begin{proof}
Let $\{u_0,\ldots,u_{n-1}\}$ denote the standard discrete Fourier basis of $\mathbb{C}^n$, given by
\[(u_k)_m=\frac1{\sqrt n}e^{2\pi i km/n},\qquad 0\ls k,m\ls n-1.\]
For $1\ls k\ls\lfloor(n-1)/2\rfloor$, define
\[v_{2k-1}=\sqrt2\,\mathrm{Re}(u_k),\qquad v_{2k}=\sqrt2\,\mathrm{Im}(u_k),\]
and set
\[v_0=u_0.\]
If $n$ is even, define in addition
\[v_{n-1}=u_{n/2}=\frac1{\sqrt n}(1,-1,1,-1,\ldots,-1).\]

The standard orthogonality relations for the discrete Fourier basis imply that these vectors form an orthonormal basis of $\mathbb{R}^n$.

Moreover,
\[(v_{2k-1})_m=\sqrt{\frac2n}\cos\!\left(\frac{2\pi km}{n}\right),\qquad (v_{2k})_m=\sqrt{\frac2n}\sin\!\left(\frac{2\pi km}{n}\right),\]
so every coordinate of $v_{2k-1}$ and $v_{2k}$ has absolute value at most $\sqrt{2/n}$. Since
\[\|v_0\|_\infty=\frac1{\sqrt n},\]
and the same estimate holds for $v_{n-1}$ when $n$ is even, every basis vector satisfies
\[\|v_i\|_\infty\ls\frac{\sqrt2}{\sqrt n},\]
as claimed.
\end{proof}

Combining the previous two results immediately yields the following consequence.

\begin{corollary}
Let $K$ be an isotropic unconditional convex body in $\mathbb{R}^n$. Then there exists an explicit orthonormal basis
$\{v_0,\ldots,v_{n-1}\}$ of $\mathbb{R}^n$ such that
\[\|\langle\cdot,v_i\rangle\|_{\psi_2}\ls C,\qquad i=0,\ldots,n-1.\]
\end{corollary}

\begin{proof}
Let $\{v_0,\ldots,v_{n-1}\}$ be the basis from Lemma~\ref{lem:fourier-basis}. By Theorem~\ref{th:unco-psi-2},
\[\|\langle\cdot,v_i\rangle\|_{\psi_2}\ls C_1\sqrt n\,\|v_i\|_\infty\ls C,\]
which proves the claim.
\end{proof}

Thus, for isotropic unconditional convex bodies, one obtains an explicit orthonormal basis of uniformly subgaussian directions with no logarithmic loss. This contrasts with the general isotropic case, where the methods developed in this work yield bounds depending on $w(\Psi_2(K))$ and $\sqrt{\ln(en)}$.

%%%%%%%%%%%%%%%%%%%%%%%%%%%%%%%%%%%%%%%%%%%%%%%%%%%%%%%%%%%%%%%%%%%%%%%%%%%%%%%%%%%%%%%%%%%%%%%%%%%%%%%%%%%%%%%%%%%%%%%%%%%%%%%
\section{Further remarks and questions}\label{section:4}
%%%%%%%%%%%%%%%%%%%%%%%%%%%%%%%%%%%%%%%%%%%%%%%%%%%%%%%%%%%%%%%%%%%%%%%%%%%%%%%%%%%%%%%%%%%%%%%%%%%%%%%%%%%%%%%%%%%%%%%%%%%%%%%

\noindent \textbf{\S~4.1. Covering numbers.} Recall that if $A$ and $B$ are two convex bodies in $\mathbb{R}^n$, then the covering number $N(A,B)$ of $A$ by $B$ 
is the least integer $N$ for which there exist $N$ translates of $B$ whose union covers $A$. Sudakov's inequality (see~\cite[Chapter~4]{AGA-book}) asserts that 
if $K$ is a symmetric convex body in $\mathbb{R}^n$, then for every $t>0$,
\begin{equation*}\ln N(K,tB_2^n) \ls n\,\frac{w(K)^2}{t^2},\end{equation*}
where $c>0$ is an absolute constant. Theorem~\ref{th:mean-width} shows that if $K$ is isotropic, then
\begin{equation*}\ln N(\Psi_2(K),tB_2^n)\ls c n\,\frac{\ln(en)}{t^2}\end{equation*}
for all $t>0$. Note that the interesting range for $t$ is $[c,C\sqrt{n}]$ because
$cB_2^n \subseteq \Psi_2(K) \subseteq C\sqrt{n}\,B_2^n$. 

Sch\"utt~\cite{Schutt-1984} has determined the covering numbers $N(\sqrt{n}B_1^n,tB_2^n)$. He showed that
\begin{equation}\label{eq:schutt}\ln N(\sqrt{n}B_1^n,tB_2^n)\ls cn\,\frac{\ln(2t)}{t^2},\qquad 1\ls t\ls \sqrt{n}/2.
\end{equation}
Since the subgaussian body of the isotropic cross-polytope is $\approx\sqrt{n}B_1^n$, this estimate suggests the question to determine if
\begin{equation}\label{eq:optimal-covering}\ln N(\Psi_2(K),tB_2^n)\ls  cn\,\frac{\ln(c_0t)}{t^2},\qquad 1\ls t\ls C\sqrt{n},\end{equation}
for every isotropic convex body $K$ in $\mathbb{R}^n$.

Let us note that, recently, Paouris and Pathak~\cite{Paouris-Pathak-2026} proved that every convex body $K$ in $\mathbb{R}^n$ has an affine image $TK$ with
$\operatorname{vol}_n(TK)=\operatorname{vol}_n(\sqrt{n}B_1^n)$ that satisfies \eqref{eq:schutt}.

\medskip

%%%%%%%%%%%%%%%%%%%%%%%%%%%%%%%%%%%%%%%%%%%%%%%%%%%%%%%%%%%%%%%%%%%%%%%%%%%%%%%%%%%%%%%%%%%%%%%%%%%%%%%%%%%%%%%%%%%%%%%%%%%%%%%
\noindent \textbf{\S~4.2. Distribution of subgaussian directions.} For every $t\gr  1$, we define
$$S_t(K):= \left\{v\in S^{n-1}:\|\langle\cdot,v\rangle\|_{\psi_2}\ls  t\right\}.$$ Consider the function
$$\psi_K(t)=\sigma(S_t(K)).$$
Set $r:=r(\Psi_2(K))$ and $w:=w(\Psi_2(K))$. An interesting question is to give a lower bound for $\psi_K(t)$, especially when $r\ll w$  
and $t\in(r,w)$.

From the proof of Theorem~\ref{th:spherical}, we see that if $t>w(\Psi_2(K))$, then
\begin{equation*}\sigma\left(\left\{v\in S^{n-1}:\|\langle\cdot,v\rangle\|_{\psi_2}>t\right\}\right)
\ls\exp\left(-\frac{(t-w(\Psi_2(K)))^2}{2C^2}\right).\end{equation*}
Hence,
$$\psi_K(t)\gr 1-\exp\left(-c_1(t-w(\Psi_2(K)))^2\right),$$
where $c_1>0$ is an absolute constant. In particular, if $t\gr  2w(\Psi_2(K))$, then $t-w(\Psi_2(K))\gr t/2$, which implies
$$\psi_K(t)\gr 1-\exp(-c_2t^2),$$
where $c_2=c_1/4$.

A subtle question is to study $\psi_K(t)$ in the case $t<w(\Psi_2(K))$. It is instructive to examine the asymptotics for the isotropic cross-polytope
$K_1={\rm vol}_n(B_1^n)^{-1/n}B_1^n$. Then, $\|\langle\cdot,v\rangle\|_{\psi_2(K_1)}\approx\sqrt{n}\|v\|_{\infty}$.  
Let
$$M_n=\sqrt{n}\,\|V\|_\infty,$$
where $V$ is uniformly distributed on the Euclidean sphere $S^{n-1}$, and define
$$\psi(t)=\mathbb{P}(M_n\ls  t).$$
One can check that
$$\left|\mathbb{E}M_n-\sqrt{2\ln n}\right|\ls C\left(\frac{\ln\ln n}{\sqrt{\ln n}}\right),$$
where $C>0$ is an absolute constant.

Let $G=(g_1,\ldots,g_n)$ be a standard Gaussian random vector. Then $V=G/|G|$ is uniformly distributed on the sphere. Hence
$M_n=\frac{\sqrt{n}\,\max_i|g_i|}{|G|}$. Direct computation shows that if $t\gg 1$, then
$$\psi(t)\gr\left(1-\frac{c}{t}e^{-t^2/2}\right)^n,$$
and hence
$$\psi(t)\gr \exp\left(-\frac{cn}{t}e^{-t^2/2}\right).$$
Writing $t=\delta\sqrt{2\ln n}$, where $\frac{1}{\sqrt{2\ln n}}<\delta<1$, we obtain
$$\psi(\delta\sqrt{2\ln n})\gr\exp\!\left(-\frac{c}{\delta\sqrt{2\ln n}}n^{1-\delta^2}\right).$$
This extremal example raises the question whether, for any isotropic convex body $K$, one can obtain an analogous lower bound for $\psi_K(t)$ 
in the regime $r(\Psi_2(K))<t<w(\Psi_2(K))$ under the assumption that $r\ll w$. 

Note that if $K$ is the isotropic Euclidean ball $D_n$ then $r=w$ and the question becomes trivial; for every $t<w$ we have $S_t(D_n)=\varnothing$.

\medskip

%%%%%%%%%%%%%%%%%%%%%%%%%%%%%%%%%%%%%%%%%%%%%%%%%%%%%%%%%%%%%%%%%%%%%%%%%%%%%%%%%%%%%%%%%%%%%%%%%%%%%%%%%%%%%%%%%%%%%%%%%%%%%%%
\noindent \textbf{\S~4.3. Gaussian width.} In this paragraph we make some observations about the Gaussian mean width of $S_t(K)$ for values of $t$ that are below $w(\Psi_2(K))$.

If $S$ is a closed subset of $S^{n-1}$, we set
$$w_G(S):=\mathbb{E}\left(\max_{x\in S}\sum_{j=1}^n x_jg_j\right),$$
where $g_j$, $1\ls  j\ls  n$, are independent standard Gaussian random variables.

Define
$$a_k:=\mathbb{E}\left[\left(\sum_{i=1}^k g_i^2\right)^{1/2}\right]=\sqrt{2}\,\frac{\Gamma((k+1)/2)}{\Gamma(k/2)}=\sqrt{k}\left(1-\frac{1}{4k}+\frac{1}{32k^2}
+O(k^{-3})\right)$$
for $k=1,\ldots,n$. Gordon's escape theorem states that if
$$a_k>w_G(S)$$
for some $k$, then there exists a subspace $F\in G_{n,n-k}$ such that
$F\cap S=\varnothing$.

\begin{proposition}\label{prop:gaussian-width}
There exists an absolute constant $C_0>0$ such that the following holds. Let $K$ be an isotropic convex body in $\mathbb{R}^n$ and define
$S_t:=S_t(K):=\left\{v\in S^{n-1}:\|\langle\cdot,v\rangle\|_{\psi_2}\ls  t\right\}$,\; $t>0$. Assume that $t\gr  2C_0$. If
$$k\gr  k_t:=\frac{4}{3}\,\left\lceil C_0^2\frac{n}{t^2}\right\rceil,$$
then, for every $E\in G_{n,k}$,
$$w_G(S_t\cap E)\gr  \frac{1}{2}\sqrt{k}.$$
In particular,
$$w_G(S_t)\gr  \frac{1}{2}\sqrt{n}.$$
\end{proposition}

\begin{proof}
We know that for every $1\ls  m\ls  n$ and any $F\in G_{n,m}$,
\begin{equation*}\operatorname{vrad}(P_F(\Psi_2(K)))\ls  C_0\sqrt{\frac{n}{m}},
\end{equation*}
where $C_0>0$ is an absolute constant. Let $t\gr  2C_0$ and define $m$ to be the smallest positive integer such that
$$C_0\sqrt{\frac{n}{m}} \leq t.$$
More precisely, set 
$$m=m(t)=\left\lceil C_0^2\frac{n}{t^2}\right\rceil.$$ Then $S_t\cap F\neq\varnothing$ for every $F\in G_{n,m}$. 

Let $k$ be large enough so that $k-m\gr  \frac{k}{4}$. Equivalently,
$$k\gr  k_t:=\frac{4}{3}\,\left\lceil C_0^2\frac{n}{t^2}\right\rceil.$$
Consider any $E\in G_{n,k}$. Then $(S_t\cap E)\cap F\neq\varnothing$ for every $m$-dimensional subspace $F$ of $E$. Therefore, Gordon's theorem yields
$$w_G(S_t\cap E)\gr  a_{k-m}\gr  \sqrt{k-m}\gr  \frac{1}{2}\sqrt{k}.$$
This proves the proposition.
\end{proof}

Combining this information with Dudley's bound (see
\cite[Exercise~8.6]{Vershynin-book}), we obtain the next fact.

\begin{proposition}\label{prop:distant-points}
With the notation of Proposition~\ref{prop:gaussian-width}, assume that
$t\gr  2C_0$ and
$$k\gr  k_t:=\frac{4}{3}\,\left\lceil C_0^2\frac{n}{t^2}\right\rceil.$$
Then, for every $E\in G_{n,k}$,
$$N(S_t\cap E,\epsilon_0 B_2^n)\gr  \exp(c_0k),$$
where $\epsilon_0,c_0$ are absolute positive constants. In particular, for every $E\in G_{n,k}$ we may find a set
$D_E\subseteq S_t\cap E$ such that
$$|D_E|\gr  \exp(c_0k)\quad\text{and}\quad |v-w|\gr \epsilon_0$$
for any pair of distinct $v,w\in D_E$.
\end{proposition}

\begin{proof}
Let $k\gr  k_t$ and consider any $E\in G_{n,k}$. We use the fact that
$$w_G(S_t\cap E)\ls C_2\int_a^b\sqrt{\ln N(S_t\cap E,\epsilon B_2^n)}\,d\epsilon,$$
where
$$a=\frac{c_1w_G(S_t\cap E)}{\sqrt{k}}\qquad\text{and}\qquad b=\operatorname{diam}(S_t\cap E).$$
Since $w_G(S_t\cap E)\gr \frac{1}{2}\sqrt{k}$ and $\operatorname{diam}(S_t\cap E)=2$, we get
$$a\gr  \epsilon_0:=\frac{c_1}{2}\qquad\text{and}\qquad b= 2.$$
It follows that
$$\frac{1}{2}\sqrt{k}\ls w_G(S_t\cap E)\ls 2C_2\sqrt{\ln N(S_t\cap E,\epsilon_0B_2^n)}.$$
This shows that
$$N(S_t\cap E,\epsilon_0B_2^n)\gr \exp(c_0k),$$
with $c_0=\frac{1}{16C_2^2}$. The second claim follows from the fact that
$$N(A,B)\ls  \mathsf{M}(A,B)$$
for any bounded $A\subset\mathbb{R}^n$ and any symmetric convex body $B$, where
$\mathsf{M}(A,B)$ is the largest cardinality of a set $D\subseteq A$ such that
$v-w\notin B$ for distinct $v,w\in D$.
\end{proof}

Thus, even when the set of subgaussian directions is potentially small in spherical measure, it has substantial Gaussian width and contains 
exponentially large separated subsets in every sufficiently large-dimensional subspace.

%%%%%%%%%%% End of paper body %%%%%%%%%%%%%%%%%%%%%%%%%%%%%%%
\bigskip

\noindent {\bf Acknowledgement.} The third named author acknowledges support by a PhD scholarship
from the National Technical University of Athens. We would like to thank Silouanos Brazitikos and Antonios Hmadi for their constructive feedback.

\bigskip

%%%%%%%%%%%%%%%%%%%%%%%%%%%%%%%%%%%%%%%%%%%%%%%%%%%%%%%%%%%%%%%%%%%%%%%%%
%%%%%%%%%%%%%%%%%%%%%%%%%%%%%%%%%%%%%%%%%%%%%%%%%%%%%%%%%%%%%%%%%%%%%%%%%

\footnotesize
\bibliographystyle{amsplain}

\bigskip

\bigskip 

\thanks{\noindent {\bf Keywords:} Hyperplane conjecture; log-concave measures; isotropic convex bodies; subgaussian directions; volume distribution in high dimensions.

\smallskip

\thanks{\noindent {\bf 2020 MSC:} Primary 52A40; Secondary  46B06, 52A23, 60D05.}

\bigskip

\bigskip

\bigskip

\noindent \textsc{Apostolos \ Giannopoulos}: School of Applied Mathematical and Physical Sciences, National Technical University of Athens, Department of Mathematics, Zografou Campus, GR-157 80, Athens, Greece.

\smallskip

\noindent \textit{E-mail:} \texttt{apgiannop@math.ntua.gr}

\bigskip

\noindent \textsc{Minas \ Pafis}: Department of Mathematics, National and Kapodistrian University of Athens, Panepistimioupolis 157-84,
Athens, Greece.

\smallskip

\noindent \textit{E-mail:} \texttt{mipafis@math.uoa.gr}

\bigskip

\noindent \textsc{Natalia \ Tziotziou}: School of Applied Mathematical and Physical Sciences, National Technical University of Athens, Department of Mathematics, Zografou Campus, GR-157 80, Athens, Greece.

\smallskip

\noindent \textit{E-mail:} \texttt{tziotziounatalia@mail.ntua.gr}

\end{document}